\documentclass[11pt]{amsart}
\usepackage{mathrsfs} 
\usepackage[margin=1.25in]{geometry}
\usepackage{ifthen}
\usepackage{graphicx}
\usepackage{epsfig}
\usepackage{dsfont, amssymb}
\usepackage{amsfonts,dsfont,amssymb,amsthm,stmaryrd,bbm}
\usepackage{amsmath}
\DeclareMathOperator*{\argmax}{arg\,max}
\DeclareMathOperator*{\argmin}{arg\,min}

\usepackage{xcolor}

\usepackage[toc,page]{appendix} 
\usepackage{hyperref,mathtools,cleveref}
\hypersetup{colorlinks=true,pdftitle=""
}

\DeclarePairedDelimiter{\norm}{\lVert}{\rVert}

\let\originallesssim\lesssim
\let\originalgtrsim\gtrsim

\DeclareRobustCommand{\lesssim}{%
  \mathrel{\mathpalette\lowersim\originallesssim}%
}
\DeclareRobustCommand{\gtrsim}{%
  \mathrel{\mathpalette\lowersim\originalgtrsim}%
}

\makeatletter
\newcommand{\lowersim}[2]{%
  \sbox\z@{$#1<$}%
  \raisebox{-\dimexpr\height-\ht\z@}{$\m@th#1#2$}%
}
\makeatother

\newtheorem{thm}{Theorem}[section]
\newtheorem{theorem}[thm]{Theorem}

\newtheorem{lem}[thm]{Lemma}
\newtheorem{lemma}[thm]{Lemma}
\newtheorem{prop}[thm]{Proposition}

\newtheorem{definition}[thm]{Definition}
\newtheorem{cor}[thm]{Corollary}

\newcommand\independent{\protect\mathpalette{\protect\independent}{\perp}} 
\def\independent#1#2{\mathrel{\rlap{$#1#2$}\mkern2mu{#1#2}}}

\DeclareMathOperator{\Var}{Var}

\DeclarePairedDelimiter{\abs}{\lvert}{\rvert}

\def\Var{{\rm Var}}

\def\phi{\varphi}

\def\bee{\begin{eqnarray*}}
\def\ene{\end{eqnarray*}}
\usepackage{mathtools}

\begin{document}

\title{Anti-concentration inequalities for log-concave variables on the real line}
\author{Tulio Gaxiola, James Melbourne, Vincent Pigno, and Emma Pollard}
\date{\today}

\begin{abstract}
    We prove sharp anti-concentration results for log-concave random variables on the real line in both the discrete and continuous setting.  Our approach is elementary and uses majorization techniques to recover and extend some recent and not so recent results.
\end{abstract}
\maketitle

\section{Introduction}
In this article we will derive some sharp inequalities between some statistical functionals on the class of log-concave probability distributions on the line.  In all cases the inequalities provide upper bounds on the amount of mass a density can assign  to a small interval of $\mathbb{R}$, and in this sense we consider these results anti-concentration inequalities.  Our first main result is the following bound on the variance of a log-concave random variable on $\mathbb{R}$.

    \begin{thm} \label{thm: variance function and distance mean to mode}
        Given $X$ a random variable with mean $\mu$ and log-concave density $f$ and $t \in \mathbb{R}$
        \[
            2\Var(X) \leq \frac{1}{f^2(t)} + (\mu - t)^2 .
        \]
    \end{thm}
    Taking $t = m$ a mode of $X$,  we have
    \[
        2 \Var(X) \leq \frac{1}{\| f\|_\infty^2} + (\mu - m )^2.
    \]
    Note that $\| f\|_\infty^{-2}$  corresponds to the $\infty$-R\'enyi entropy power, sometimes written $N_\infty(X)$; see \cite{MMX17:0} for background.
    When $X$ is symmetric, necessarily $\mu = m$, and this recovers a result of Hensley \cite{hensley1980slicing}, and should be compared to Bobkov \cite{bobkov1999isoperimetric}, who proved that $2 \Var(X) \leq N_\infty(X)$ holds when the median of $X$ matches $\mu$.
    Taking $t = \mu$, we have
    \[
        \Var(X) f^2(\mu) \leq \frac 1 2.
    \]
    These results follow from a reduction through a stochastic ordering argument (see \cite{marshall2011inequalities, shaked2007stochastic}) to the asymmetric Laplace distributions, which attain equality in Theorem~\ref{thm: variance function and distance mean to mode}.  The use of majorization in congruence with log-concavity dates back at least to \cite{whitt1985uniform}, see also \cite{yu2008maximum, yu2009entropy, marsiglietti2022moments, melbourne2020reversal} for more recent developments.
    From here we analyze in detail the asymmetric Laplace distributions. Here our work is essentially a reformulation of Fradelizi \cite{fradelizi1999hyperplane}, who considered the problem in the context of hyperplane slices of convex bodies.  To state our result we will use the following definitions and notations. 
 For $W$ a random variable with density $f$ with respect to a reference measure, which in this paper will be either the Lebesgue measure on $\mathbb{R}$ or the counting measure on $\mathbb{Z}$, $M(W)$ denotes the essential supremum of $f$ with respect to said reference measure.  For $\psi$, a Young function, that is $\psi$ strictly increasing, convex, with $\psi(0) = 0$, we define the $\psi$-Orlicz norm 
    \[
        \| W \|_\psi \coloneqq \inf \left\{ t > 0 : \mathbb{E} \psi \left( \frac{|W|}{t} \right) \leq 1 \right \}. 
    \]
    \begin{thm} \label{thm: orlicz norm}
    For $X$ log-concave, $U$ uniform on an interval, and $Z$ exponential with parameters determined implicitly by the equality $M(U) = M(X) = M(Z)$, 
    \[
        \| U - \mathbb{E}U \|_\psi \leq \| X - \mathbb{E}{X} \|_{\psi} \leq \| Z - \mathbb{E}Z \|_\psi.
    \]
\end{thm}
See \cite{leonard2007orlicz, rao1991theory} for background on Orlicz spaces. An important example is given by the case $\psi_p(x) = x^p$ for $p \geq 1$, in which case $\|W\|_{\psi_p} = \left(\mathbb{E}|W|^p \right)^{\frac 1 p}$. In particular from Theorem \ref{thm: orlicz norm} we obtain, as  Corollary \ref{cor: maximums to ACM},  sharp upper and lower bounds on $M(X)$ in terms of the $p$-absolute moment, for $p \geq 1$, extending Bobkov and Chistyakov's result in \cite{BC15:2} for the variance. 

In the discrete setting we obtain an analogous result through an analogous reduction to a class of extremal distributions with fixed maximum value, analogous to the asymmetric Laplace in the continuous setting exists, which we will refer to as discrete asymmetric Laplace distributions when there is risk of ambiguity with the continuous setting.  We obtain a discrete analog of Theorem \ref{thm: variance function and distance mean to mode}.  
\begin{thm} \label{thm: discrete variance point inequality}
    For $Y$ a log-concave random variable on $\mathbb{Z}$, and $n \in \mathbb{Z}$,
    \[
       2\  \Var[Y] \leq \left( \frac{1}{\mathbb{P}^2(Y=n)} - 1 \right) + (\mu - n)^2 .
    \]
    with equality if and only if $Y$ is a discrete asymmetric Laplace with mode $n$.
\end{thm}

Taking $n = m \coloneqq m_Y$ defined to be a mode of $Y$, we obtain
\[
   2 \ \Var[Y] \leq  \left( \Delta_\infty(Y) + (\mu - m)^2 \right)
\]
where $\Delta_\infty(Y) \coloneqq \frac{1}{M^2(Y)} - 1$ as in \cite{madiman2023bernoulli}.
Note that if the mean of $Y$ is a mode, this reduces to the inequality $\Var[Y] \leq \frac 1 2 \Delta_\infty(Y)$, which implies the inequality of Bobkov-Marsiglietti-Melbourne \cite{bobkov2022concentration}, where the same conclusion is obtained under the assumption that $Y$ is symmetric (which implies that $\mu = m = 0$).  
Note, the classical result of Darroch \cite{darroch1964distribution} states that when $Y$ has the distribution of an independent Bernoulli sum, at least one of $\lfloor \mu \rfloor$ and $\lceil \mu \rceil$ is a mode, which yields $2 \Var[Y] \leq \Delta_\infty(Y) + 1$.  
We note that the constant $1$ can be removed under these additional assumptions (see \cite{madiman2023bernoulli}). {\color{black} 
 If we take $n = [ \mu ] \coloneqq \argmin \{ m \in \mathbb{Z} : | \mu - m | \}$ we obtain from $| \mu - [\mu]|^2 \leq \frac 1 4$ that
\begin{align*}
     \mathbb{P}^2(Y = [\mu]) \leq \frac{1}{\frac 3 4 + 2 \Var[Y]}.
\end{align*}
In the case that $\mu \in \mathbb{Z}$ this can be improved to 
$
     \mathbb{P}^2(Y = \mu) \leq (1 + 2 \Var[Y])^{-1}.
$
}

However, we are unable to reach an analog of Theorem \ref{thm: orlicz norm} in the discrete setting.  The failing stems from the fact that an analog of \Cref{lem:sublevel_set_formula} below does not hold.  That is, in the continuous setting, all asymmetric Laplace distributions with a fixed maximum are equimeasurable, while in the discrete setting this is obviously false.  We do however, obtain the following.

\begin{thm} \label{thm: discrete piotr theorem}
    For $X$ a log-concave random variable on $\mathbb{Z}$ with $M(X) \coloneqq \max_{n \in \mathbb{Z}} \mathbb{P}(X = n)$,
    \begin{align} \label{eq: Maximum variance discrete LC}
        M^2(X) \Var(X) + M(X) \leq 1,
    \end{align}
and
        \begin{align} \label{eq: Maximum 4 ACM discrete LC}
       {\color{black} M^4(X) \ \mathbb{E} |X - \mathbb{E}[X]|^4 + M(X) \big( M^2(X) - 10 M(X) + 18 \big) \leq 9.}
    \end{align}
    The equality cases of \eqref{eq: Maximum variance discrete LC} and \eqref{eq: Maximum 4 ACM discrete LC} are identical, occuring if and only if $X$ is, up to translation and reflection, a geometric random variable.
\end{thm}
Inequality \eqref{eq: Maximum 4 ACM discrete LC} is new, extending  \eqref{eq: Maximum variance discrete LC} which is due to Jakimiuk-Murawski-Nayar-Slobodianiuk \cite{jakimiuk2024log}, who improved the asymptotically sharp result of Aravinda \cite{aravinda2024entropy}, whose results improved \cite{bobkov2022concentration}.  In all three of the mentioned papers the inequality is used to provide upper bounds on the Levy concentration function of a log-concave random variable.  Our   proofs are considerably shorter than  \cite{aravinda2024entropy, jakimiuk2024log} and avoid the recently developed discrete ``localization'' or ``degrees of freedom'' machinery, see also \cite{alqasem2024conjecture, aravinda2024entropy, aravinda2022concentration,  jakimiuk2024log, marsiglietti2024geometric} for development and recent work.  See \cite{payne1960optimal, lovasz1993random, kannan1995isoperimetric, fradelizi2004extreme, fradelizi2006generalized, eldan2013thin, bobkov2016hyperbolic, klartag2017needle} for the development of the analogous technique in continuous settings.  Moreover, a characterization of equality cases is easily given in the process.

Let us outline the contents of the paper. 
In \Cref{sec: defn} we give definitions and state some relevant lemmas.  
In particular we state a pair of stochastic ordering results related to the ``crossings'' of densities.  
General results of this style go back at least to \cite{karlin1963generalized}.  
\Cref{sec: logconcave cont}  contains the proofs of the anticoncentration inequalities in the continuous setting, while the discrete setting is pursued in \Cref{sec: discrete log-concave}.  
Finally, we give a general density crossings - majorization type result in \Cref{sec: Majorization}. 
This result is not new and is actually less general than that of \cite{karlin1963generalized}, but our proof is perhaps easier for our reader.  At minimum it is included for convenience and completeness.

\section{Definitions} \label{sec: defn}
A function $f: \mathbb{R} \to [0,\infty)$ is log-concave if
\[
    f((1-t) x + ty) \geq f^{1-t}(x)f^t(y)
\]
holds for all $x,y \in \mathbb{R}$ and $t \in [0,1]$.
For a Borel random variable $X$ and probability measure $\nu$ we write $X \sim \nu$ to denote that for all Borel $A$
\[
    \nu(A) = \mathbb{P}(X \in A).
\]
We consider a random variable $X \sim \nu$ to be log-concave with respect to a measure $\mu$, if $\frac{d\nu}{d\mu}$ exists and there is a log-concave function $f$ (called the density) such that 
\[
    f = \frac{d \nu}{d \mu} \hspace{7.5 mm} \hbox{$\mu$-almost surely.}
\]
By a slight abuse of notation, for variables with a density function $f$ we will write $X \sim f$, in the case that $X \sim \nu$ and $f = \frac{d\nu}{d\mu}$.
Recall that for an r.v. $X \sim f$ we denote the essential supremum of its density with repsect to $\mu$ by $M(X) = \norm{f}_{\infty}$. 
Throughout this paper we only consider the case that $\mu$ is the Lebesgue measure on $\mathbb{R}$ or the counting measure on $\mathbb{Z}$.

\begin{definition}
    Given a class of functions $\mathcal{A}$ mapping from a Borel $S \subseteq \mathbb{R}$ to $\mathbb{R}$ we write $X \prec_{\mathcal{A}} Z$ for $S$-valued random variables $X,Z$ such that 
    \[
        \mathbb{E}[f(X)] \leq \mathbb{E}[f(Z)]
    \]
    holds for every $f \in \mathcal{A}$.
\end{definition}
In particular:
\begin{enumerate}
\item When $\mathcal{A}$ is the class of functions from $\mathbb{R}$ to $\mathbb{R}$ with non-negative  $n$th derivative, we write $\prec_n$. 
\item When $\mathcal{A}$ is the class of convex functions, we write $\prec_{cx}$. 
\item When $\mathcal{A}$ is the class of all non-decreasing convex functions, we write $\prec_{icx}$.
\end{enumerate}
To be more precise, we wish to say that $X \prec_n Z$ means that $\mathbb{E}[f(X)] \leq \mathbb{E}[f(Z)]$ for any $f$ with {\it distributional} $n$-th derivative a non-negative measure.  This can be approached classically, by saying that $f$ is increasing in the $n =1$ case, and that for $n \geq 2$, the $(n-2)$th (ordinary) derivative of $f$ is convex, see \cite[page 54]{schwartz1951theorie}. 
We direct the interested reader to \cite{marshall2011inequalities, shaked2007stochastic} for further background.

Let us also collect a pair of elementary observations. The first says that the size of the intersection of two intervals gets smaller as the distance between their barycenters gets larger.  The second is an obvious corollary of the fact that convex functions have convex sub-level sets.  We use here and throughout the notation $|A|$ to denote the Lebesgue measure of a measurable set $A$.  We also use $a \wedge b \coloneqq \min \{ a, b\}$.
\begin{lem} \label{lem: translating intervals}
For $a,b > 0$, the function $f(x) = |[-a,a] \cap [x -b, x+b]|$ can be written as 
\[
    f(x) 
    = 
        \begin{cases}
            2 \ (a \wedge b) & |x| \leq |b-a| \\
            a+b - |x| & |x| \in [ |b-a|, a+b] \\
            0 & |x| \geq a + b. 
        \end{cases}
\]
 In particular $x \mapsto |[-a,a] \cap [x -b, x+b]|$ is a non-increasing function of the absolute value of $x$ while $x \mapsto |[-a,a]^c \cap [x -b, x+b]| = 2 b - f(x)$ is non-decreasing function of the absolute value of $x$.
 \end{lem}

 \begin{lem} \label{lem: convex functions are not rascals}
    Given a convex function $f: [0,\infty) \to \mathbb{R} \cup \{ \infty \}$ such that $f(0) = 0$, $\{ f \leq 0 \}$ is an interval containing $0$.
 \end{lem}

Our last results for this section are the density crossing to majorization results for densities.

\begin{lemma}\label{lem:one_crossing}
        Suppose that $\phi \in L^1(\mathbb{R}, \mu)$ is such that $\int \phi d\mu = 0$ and there exists $x'$ such that $\phi(x) \leq 0$ for $x \leq x'$ and $\phi(x) \geq 0$ for $x > x'$. Then for $f$ non-decreasing,
    \[
        \int f \phi d\mu \geq 0.
    \]
       Moreover, if there exists $f$ strictly increasing (in the sense that $f(x) < f(y)$ for $x < y$) 
       such that $\int f \phi d\mu = 0$, then $\phi = 0$ $\mu$-almost surely.
\end{lemma}

We note that when $\phi = g_2 - g_1$ is the difference between two probability density functions one automatically obtains $\int \phi d\mu = 0$. 
Thus the existence of an $x'$ above implies that for $X_1 \sim g_1$ and $X_2 \sim g_2$, that $X_1 \prec_1 X_2$.
Additionally if $\mathbb{E}[f(X_1)] = \mathbb{E}[f(X_2)]$ for a strictly increasing function $f$ (for instance, $f(x) = x$ will be used below) then $g_1 = g_2$ $\mu$-almost surely.

 \begin{lemma}\label{lem:two_crossing}
    Suppose that $\phi \in L^1(\mathbb{R}, \mu)$ is such that $\psi(x) \coloneqq x \phi(x) \in L^1(\mathbb{R}, \mu)$ and satisfies $\int \psi d\mu = \int \phi d\mu = 0$, while there exists $x_0 < x_1$ such that $\phi(x) \leq 0$ for $x \in (x_0, x_1)$ and $\phi(x) \geq 0$ for $x \notin [x_0,x_1]$. 
    Then for $f$ convex,
    \[
        \int f \phi d\mu \geq 0.
    \]
    Moreover, if there exists $f$ strictly convex (in the sense that $f((1-t) x + ty) < (1-t) f(x) + t f(y)$ for $t \in (0,1)$ and $x \neq y$) 
    such that $\int f \phi d\mu = 0$, then $\phi = 0$ $\mu$-almost surely.
\end{lemma}

Again, when $\phi = g_2 - g_1$ for densities that ``cross-twice'' in the sense of the existence of $x_1$ and $x_2$ as above, then for random variables $X_i \sim g_i$, if $\mathbb{E}X_1 = \mathbb{E}X_2$ then one has $X_1 \prec_{cx} X_2$. 
Moreover, equality as in $\mathbb{E}f(X_1) = \mathbb{E}f(X_2)$ for strictly convex $f$ implies $g_1 = g_2$ $\mu$-almost surely.

\section{Lebesgue measure} \label{sec: logconcave cont}
In this section we consider random variables that are log-concave with respect to the Lebesgue measure on $\mathbb{R}$, and derive results for the interaction between their variance and mean with their density maximums and mode.  
Within this section, the term ``log-concave random variables'' will always refer to log-concave random variables with respect to the Lebesgue measure. 
Let us begin by defining asymmetric Laplace densities and deriving formulas for their moments. 

\begin{definition} \label{def: asymmetric Laplace Lebesgue}
     For $\lambda \in [0,\infty)^2 - \{0\}$ we define a mode-centered asymmetric Laplace density by
    \[
        g_{\lambda}(x) = \frac{e^{- |x|/\lambda_1} \mathbbm{1}_{(-\infty,0]}(x) 
        +  e^{- x / \lambda_2}  \mathbbm{1}_{(0,\infty)}(x) }{\lambda_1 + \lambda_2} . 
    \]
    when $\lambda_i \in (0,\infty)$.
    We include the Laplace distribution as the special case $\lambda_1=\lambda_2$.
    For $\lambda_1 > 0$, we define 
    $$
    g_{(\lambda_1, 0)}(x) = \frac{e^{- |x|/\lambda_1} \mathbbm{1}_{(-\infty,0]}(x) }{\lambda_1},
    $$
    and analogously, for $\lambda_2 > 0$, we define
    $$
    g_{(0, \lambda_2)}(x) \coloneqq \frac{e^{-x/\lambda_2} \mathbbm{1}_{[0,\infty)}(x)}{\lambda_2}. 
    $$ 
    
    We call a random variable asymmetric Laplace when its density function with respect to the Lebesgue measure can be written as $x \mapsto g_{\lambda}(x - m)$ for $g_\lambda$ the mode-centered asymmetric Laplace density associated to some $\lambda$ and $m \in \mathbb{Z}$. 
\end{definition}

Observe that $\norm{g_\lambda}_{\infty} = \frac 1 {\lambda_1 + \lambda_2}$.

\begin{lem}\label{lem: asymmetric Laplace Lebesgue moments}
    For $X_\lambda$ an asymmetric Laplace random variable with mode at $0$,
    \[
        \mathbb{E}[e^{t X_\lambda}] = \frac{1}{(1 + t \lambda_1)(1 - t \lambda_2 )} = 1 + t (\lambda_2 - \lambda_1) + t^2 ( \lambda_2^2 - \lambda_1 \lambda_2 + \lambda_1^2) + O(t^3).
    \]
    In particular
    \begin{align*}
        \mathbb{E}[X_\lambda] &= \lambda_2 - \lambda_1
            \\
        \Var[X_\lambda] &= \lambda_2^2 + \lambda_1^2.
    \end{align*}
\end{lem}

\begin{proof}
    Direct computation.
\end{proof}

The key technical result is the following.

\begin{prop} \label{prop: match a point and convex domination}
    Given a point $t \in \mathbb{R}$ and a log-concave random variable $X \sim f$, there exists an asymmetric Laplace random variable $X_\lambda \sim f_\lambda$ such that \begin{align*}
        f(t) = f_\lambda(t) \qquad \text{ and } \qquad X \prec_{cx} X_\lambda.
    \end{align*}
\end{prop}

\begin{proof}
    By the translation invariance of the problem, it suffices to consider $t = 0$, and so we write $X \sim g$ and $X_\lambda \sim g_\lambda$ to match the notation of \Cref{def: asymmetric Laplace Lebesgue}.  
     Taking $\lambda_2 = \frac 1 {g(0)} - \lambda_1$ we have a class of asymmetric Laplace densities that satisfy $g_\lambda(0) = g(0)$. 
     Define convex functions $f_i : [0,\infty) \to \mathbb{R}\cup \{ \infty \}$ by
     \begin{align*}
        f_1(x) \coloneqq \log g_\lambda(x) - \log g(x)
        \qquad \hbox{ and } \qquad
        f_2(x) \coloneqq \log g_\lambda(-x) - \log g(-x)
     \end{align*}
     so that $f_i(0) = 0$.  
     By \Cref{lem: convex functions are not rascals}, 
     $I_i = \{ f_i \leq 0 \}$ are both intervals containing the origin, and hence
     \begin{equation}\label{eq:at_most_two_crossings}
        \{ g_\lambda \leq g \} = I_1 \cup (-I_2).
     \end{equation}
     is an interval.  
     Define $\lambda(t) = (t, \frac 1 {g(0)} - t)$ 
     and define $\Phi: [0, 1/g(0)] \to \mathbb{R}$ by 
     \[ \Phi(t) = \mathbb{E} X_{\lambda(t)} = \frac{1}{g(0)} - 2 t. \]
     Note that $\Phi$ is continuous; 
     we will show that $\Phi(0) \geq \mathbb{E}[X] \geq \Phi(1/g(0))$ 
     and apply the intermediate value theorem. 
     
     Indeed, $g_{\lambda(0)} = 0\leq g$ on $(-\infty,0]$ and hence 
     \[
        \{ g_{\lambda(0)} \leq  g \}
     \]
     is an infinite interval containing all negative numbers. 
     It follows from 
     \Cref{lem:one_crossing}
     that $X \prec_1 X_{\lambda(0)}$, 
     and in particular $\mathbb{E}X \leq \mathbb{E}X_{\lambda(0)} = \Phi(0)$. 
     Similarly, $g_{\lambda(g^{-1}(0))} = 0$ on $[0, \infty)$ and hence
     \[
        \{ g_{\lambda(g^{-1}(0))} \leq  g \}
     \]
     is an infinite interval containing containing all positive numbers.  
     Thus,
     $X_{\lambda(g^{-1}(0))}\prec_1 X$, and $\Phi(1/g(0)) \leq \mathbb{E}[X]$.  
     Thus by the intermediate value theorem there exists $t$ such that 
     \[
        \mathbb{E}[X_{\lambda(t)}] = \mathbb{E}[X] .
     \]
     
     As shown in \eqref{eq:at_most_two_crossings}, the densities $g_{\lambda(t)}$ and $g$ can cross at most two times.
     If they cross exactly once (i.e. exactly one of the $I_i$ is infinite), then 
     \Cref{lem:one_crossing}
     contradicts the equality of the means; 
     hence the densities must either cross twice (i.e. the $I_i$ must both finite) or never cross (i.e. $I_1 = I_2 = [0, \infty)$). 
     If the densities cross twice, 
     \Cref{lem:two_crossing}
     shows that $X \prec_{cx} X_{\lambda(t)}$;
     and if the densities never cross, then $g_{\lambda(t)} = g$ and $X \prec_{cx} X_{\lambda(t)}$ is trivial.
\end{proof}

We can now prove \Cref{thm: variance function and distance mean to mode}, which we restate for the convenience of the reader.
\begin{theorem}
  Given $X$ a random variable with mean $\mu$ and log-concave density $f$ and $t \in \mathbb{R}$
  \[
      2\Var(X) \leq \frac{1}{f^2(t)} + (\mu - t)^2 
  \]

\begin{proof}
  It suffices to consider $t$ belonging to the support of $f$, else $\frac{1}{f^2(t)} = \infty$.  As the inequality is translation invariant, we assume without loss of generality that $t = 0$.  Taking $X_\lambda$ as in \Cref{prop: match a point and convex domination}, we necessarily have  $\lambda_2 = \frac 1 2 \left( \frac 1 {f(0)} + \mu \right)$ and $\lambda_1 = \frac 1 2 \left( \frac 1 {f(0)} - \mu \right)$, $f_{X_\lambda}(0) = f(0)$, $\mathbb{E}[X_\lambda] = \mathbb{E}[X]$, and $X \prec_{cx} X_\lambda$.  In particular 
  \[
  \frac{1}{2} \left( \frac 1 {f^2(0)} + \mu^2 \right) = \lambda_1^2 + \lambda_2^2 = \Var(X_\lambda) \geq \Var(X),
  \]
  which yields our result.
\end{proof}
\end{theorem}

In what follows we pursue inequalities for log-concave random variables that are independent of the position of the mean.

{\color{black}
\begin{lemma}\label{lem:sublevel_set_formula}
The $t$-superlevel set of the density $g_{\lambda}$ of a mode-centered asymmetric Laplace $X_\lambda$ with , is given by the formula
\[ \{ g_{\lambda} > t \} = \log((\lambda_1 + \lambda_2)t)\cdot(-\lambda_1,\lambda_2).  \]
In particular, 
$$
    | \{ g_\lambda > t \}| = \frac{1}{M} \log \frac{t}{M},
$$
where $M = M(X_\lambda)$, and thus  asymmetric Laplace densities with $\phi_{\lambda,w} = g_{\lambda}(\,\cdot-w)$ with the same maximum value are equimeasurable. 
\end{lemma}

\begin{proof}
The computation of the superlevel set is elementary. 
To see equimeasurability, note that 
translation of $g_{\lambda}$ does not change the measure of its superlevel sets 
and that we have fixed the quantity $M(X_{\lambda}) = \frac{1}{\lambda_1 + \lambda_2}$. 
\end{proof}

Equimeasurability is the foundation of the following technical lemma. 
\begin{lemma}\label{lem: asymmetric Laplace suplevel sets}
Fix constants $a,M \ge 0$ and let $X \sim \phi$ be asymmetric Laplace with mean $0$ and $M(X) = M$. 
Then, the quantity $\abs{[-a,a]^c\cap\{\phi > t\}}$ depends only on $\lambda_2$, is minimized by taking $\lambda_2 = \frac{1}{2M}$, and is maximized by taking $\lambda_2 = \frac{1}{M}$. 
\end{lemma}
}
\begin{proof}
We may write $X \sim \phi = g_{w + \lambda_2, \lambda_2} (\,\cdot - w)$ for $w = \frac 1 M - 2 \lambda_2$ and $\lambda_2 \in \left[0, \frac 1 M \right]$. 
Indeed, by the formulas for mean and maximum of asymmetric Laplace densities, $w, \lambda_1, \lambda_2$ must satisfy
$\lambda_1 = \lambda_2 + w$ and $\frac 1 M = \lambda_1 + \lambda_2$.
\Cref{lem:sublevel_set_formula} shows that for $t < M$, 
\[
  \{ \phi > t \} 
     = \lambda_2 (R(t)-1) + \left( - \frac {R(t)} M, \frac 1 M \right),   
\]
where we have written $R(t) = \log \left( \frac M t \right) -1$.
Moreover, for $t \leq M e^{-2}$, $R(t) - 1 \geq 0$ and
\[
  \left\{ \frac{1}{M} \right\} 
    = \argmax_{\lambda_2 \in [0,\frac{1}{M}]} \bigg( \sup \{ \phi > t \} \bigg), 
\]
while for $t \geq M e^{-2}$, $R(t) - 1 \leq 0$ and 
\[
  \left\{ \frac{1}{M} \right\} 
    = \argmin_{\lambda_2 \in [0,\frac{1}{M}]} \bigg( \inf \{ \phi > t \} \bigg).
\]
\Cref{lem: translating intervals} now shows that $\lambda_2 = \frac{1}{M}$ (or $\lambda_2 = 0$) maximizes the objective function. 

Also, taking $\lambda_2 = \frac 1 {2M}$ yields the symmetric interval
\[
  \{ \phi > t \}
  = \{ g_{(\frac{1}{2M}.\frac 1{2M})} > t \} 
  = \left( - \frac{R(t) +1 }{2M}, \frac{R(t)+1}{2M} \right).
\]
Another application of \Cref{lem: translating intervals} now 
completes the proof. 
\end{proof}

 \begin{cor} \label{cor: increasing order for asymmetric Laplace}
     For $U$ uniform on an interval, 
     $S$ Laplace, 
     $X$ asymmetric Laplace, 
     and $Z$ or $-Z$ exponential,
     \[
        |U- \mathbb{E}U| \prec_1 |S - \mathbb{E}S| \prec_1 |X - \mathbb{E}X| \prec_1 |Z - \mathbb{E} Z |,
     \]
     when $0 < M = M(U) = M(S) = M(X) = M(Z)$.
 \end{cor}

Note that $|U - \mathbb{E}U| \prec_{1} |X - \mathbb{E}X|$ holds for general $X$ with density bounded by $M$.
 \begin{proof}
     For the first inequality, note that if $X$ has density $f$, then $|X|$ has density $x \mapsto (f(x) + f(-x))\mathbbm{1}_{[0,\infty)}(x)$, so that $|U-\mathbb{E}U|$  has the density function $\phi(x) = 2 M \mathbbm{1}_{[0, 1/2M]}(x)$.  For $X$ with density bounded by $M$, $|X|$ has density $(f(x) + f(-x))\mathbbm{1}_{[0,\infty)}(x)$ bounded by $2M$, $\phi = 2M \geq f $ on $[0,\frac{1}{2M}]$, while trivially $\phi \leq f$ on $(\frac{1}{2M}, \infty)$.  Thus by \Cref{thm: Karlin-Novikoff criteria} $|U- \mathbb{E}U| \prec_1 |S - \mathbb{E}S|$ holds.

    For the remaining cases we note that for $\psi$ non-decreasing and $X$ a random variable with density $f$
    \begin{align*}
        \mathbb{E} \psi(|X|) = \int_{\mathbb{R}} \psi(|x|) f(x) dx 
        &= \int_0^\infty \int_0^\infty \left(\int_{\mathbb{R}} 
          \mathbbm{1}_{\{ w: \psi(|w|) > t \}}(x) 
          \mathbbm{1}_{ \{ w : f(w) > s \}}(x) 
          dx \right) dt ds \\
        &= \int_0^\infty \int_0^\infty \left| \{ |w| \in \psi^{-1} (t, \infty) \}\cap \{ f > s \} \right|  dt ds
    \end{align*} 
    Observe that since $\psi$ is non-decreasing $$\big|\{ |w| \in \psi^{-1} (t, \infty) \} \cap \{ f > s \} \big| = \big| ((-\infty, a] \cup [a, \infty) ) \cap \{ f > s  \} \big|$$
    for $a = \inf \{x: \psi(x) > t\}$, and it suffices to prove 
    \[
       |[-a,a]^c \cap \{ \phi_{\Lambda(1/2M)} > t \}| 
       \leq |[-a,a]^c \cap \{ \phi_{\Lambda(\lambda)} > t \}| 
       \leq |[-a,a]^c \cap \{ \phi_{\Lambda(1/M)} > t \}|.
    \]
    This is the content of \Cref{lem: asymmetric Laplace suplevel sets}.
    \end{proof}

Now we give a slightly expanded version of \Cref{thm: orlicz norm} and its proof.

\begin{thm}\label{thm: ordering with fixed maximum}
    For $X$ a random variable with $M \coloneqq M(X) < \infty$, 
    \[
        U \prec_{cx} X - \mathbb{E} X,
    \]
    for a uniform distribution $U \sim f_U(x) \coloneqq M \mathbbm{1}_{[-1/2M, 1/2M]}(x) $.
    Moreover, if $X$ is log-concave,
    \[
        |X - \mathbb{E} X| \prec_{icx} | Z - \mathbb{E} Z |
    \]
    for an exponential distribution $Z \sim f_Z(x) \coloneqq M e^{-M x} \mathbbm{1}_{(0,\infty)}(x).$ 
\end{thm}

In particular, for $\psi$ a Young function defining the Orlicz norm $\| \cdot \|_\psi$,
    \[
        \| U - \mathbb{E}U \|_\psi \leq \| X - \mathbb{E}{X} \|_{\psi} \leq \| Z - \mathbb{E}Z \|_\psi.
    \]

\begin{proof}
For the first inequality, observe that $f_U \geq M \geq f_X$ on $[-\frac 1 {2M}, \frac{1}{2M}]$, and trivially $0 = f_U \leq f_X$ on the complement.  Hence by \Cref{thm: Karlin-Novikoff criteria} we have $U \prec_{cx} X - \mathbb{E}X$.

  For the other inequality, 
  \Cref{prop: match a point and convex domination} shows that 
  given $X$, there exists asymmetric Laplace $X_\lambda$ with the same mean $\mu$ such that $X \prec_{cx} X_\lambda$. 
  Hence for a convex increasing function $\psi:[0,\infty) \to [0,\infty)$ such that $\psi(0) = 0$, $\phi(x) = \psi(|x - \mu|)$ is convex, and
    \[
        \mathbb{E}\phi(X) = \mathbb{E} \psi(|X - \mathbb{E}X|) \leq \mathbb{E} \psi(|X_\lambda - \mathbb{E} X_\lambda|) = \mathbb{E}\phi(X_\lambda).
    \]
    Applying \Cref{cor: increasing order for asymmetric Laplace}, since $\psi$ is increasing we have
    \[
        \mathbb{E} \psi(|X_\lambda - \mathbb{E} X_\lambda|) \leq \mathbb{E} \psi(|Z - \mathbb{E} Z|). \qedhere
    \]
\end{proof}

For a random variable $X$ with finite expectation $\mathbb{E}[X]$ and $p \geq 1$, denote the $p$-th absolute central moment of $X$, $$\sigma_p(X) \coloneqq \mathbb{E}|X - \mathbb{E}[X]|^p.$$  We will also use the following notation for the ``subfactorial'' of an integer $n \geq 1$,
\[
    !n \coloneqq \int_0^\infty (x-1)^n e^{-x} dx.
\]
Combinatorially $!n$ corresponds to the number ``derangements'' of $n$ elements, that is the permutations $\phi$ of $n$ elements such that $\phi(j) \neq j$ for all $j$; see for instance \cite{stanley2011enumerative}. 
\begin{cor} \label{cor: maximums to ACM}
    For $X$ log-concave and $p \geq 1$,
    \[
           {\color{black} \frac{1}{2^p(p+1)} \leq M^p(X) \sigma_p(X) \leq \frac{\Gamma(1+p)}{e} +
            \int_0^1 (1-x)^p e^{-x} dx.}
    \]
    where the left hand inequality holds for any random variable $X$, and with equality when the random variable is uniform on an interval, while the right hand side holds with equality for an exponential distribution.  Moreover for integers $n$ we have
    \[
        M^{n}(X) \sigma_n(X) \leq \frac{n!}{e} + (-1)^n \left( !n - \frac{n!}{e} \right).
    \]
\end{cor}

Note that for even integers $p = 2n$, one has the inequality 
\[
    M^{2n}(X) \sigma_{2n}(X) \leq \  !(2n).
\]
\begin{proof}
    $M^p(X)\sigma_p(X)$ is affine invariant, so it suffices to prove the inequalities when $M(X) =1$, in which case the left hand inequality follows from Theorem \ref{thm: ordering with fixed maximum} applied to the convex function $\phi_p(x) = |x|^p$ and the ordering $U \prec_{cx} X - \mathbb{E}X$, where $U \sim f_U(x) = \mathbbm{1}_{[-1/2,1/2]}(x)$ so that
    \[
        \frac{1}{2^p(p+1)} = \mathbb{E}|U|^p \leq \mathbb{E}|X - \mathbb{E}[X]|^p = M^p(X) \sigma_p(X).
    \]
    With $f_Z(x) = e^{-x}\mathbbm{1}_{[0,\infty)}(x)$, and the increasing convex function $\Phi_p:[0,\infty) \to [0,\infty)$, $\Phi_p(x) = x^p$ applying the other majorization result of Theorem \ref{thm: ordering with fixed maximum}, $|X - \mathbb{E}[X]| \prec_{icx} |Z - \mathbbm{E}[Z]|$, that
    \begin{align*}
        M^p(X) \sigma_p(X)
            \leq 
                \int_0^\infty |x-1|^p e^{-x} dx = \int_1^\infty (x-1)^p e^{-x} dx + \int_0^1 (1-x)^p e^{-x} dx,
    \end{align*}
    recovering the upper bound after the subsitution $u = x-1$ in the first integral.  For $n$ an integer,
    \begin{align*}
                \int_0^1 (1-x)^n e^{-x} dx
            &=
                (-1)^n \left( \int_0^\infty (x-1)^n e^{-x} dx - \int_1^\infty (x-1)^n e^{-x} dx \right) 
                    \\
            &=
                (-1)^n \left( !n - \frac{n!}{e} \right) . \qedhere
    \end{align*}
\end{proof}
For example when $p =2$, as there is exactly $1$ derangement of a set of two elements, we recover
\[
    \frac 1 {12} \leq  M^2(X) \Var(X) \leq 1, 
\]
of Bobkov and Chistyakov \cite[Proposition 2.1]{BC15:2}, which was used to obtain the  super-additivity properties of the L\'evy concentration function of sums of independent random variables (see Theorem 1.2 of the same paper).

\section{The Counting Measure} \label{sec: discrete log-concave}
In this section our reference measure is the counting measure on $\mathbb{Z}$, and a ``log-concave'' random variable in this section will always mean a random variable with a log-concave density function with respect to the counting measure on $\mathbb{Z}$.
\begin{definition}
    A mode-centered asymmetric Laplace distribution on $\mathbb{Z}$ is a probability distribution of the form
    \[
        g_{p,q}(n) = C_{p,q}  \begin{cases} 
                            p^{|n|} &\hbox{ for integers } n \leq 0,
                                \\
                            q^n &\hbox{ for integers } n \geq 0
                        \end{cases}
    \]
    for $p,q \in [0,1)$, with $C_{p,q} = \frac{(1-p)(1-q)}{1-pq}$.  
\end{definition}

In this section we will only consider integer valued random variables, and for brevity we will drop the suffix, ``on $\mathbb{Z}$''.  As such We consider a random variable $Z \sim \phi$ to be asymmetric Laplace if $\phi(n) = g_{p,q}(n - m)$  for some $m \in \mathbb{Z}$ and $g_{p,q}$ a mode-centered asymmetric Laplace probability sequence.  We take the case $p = q = 0$ to correspond to a point mass at $0$.  For $p = 0$ and $q > 0$ we recover the geometric distribution and with $p > 0$ and $q = 0$ the reflection of a geometric distribution, both as mode-centered asymmetric Laplace distributions.

{
\begin{lem}
    For $X \coloneqq X_{p,q} \sim g_{p,q}$,
    \[
        \mathbb{E}[ e^{tX}] = \frac{(1-p)(1-q)}{(e^t - p)(e^{-t} -q)}
    \]
    for $|t|$ sufficiently small. 
    In particular,
    \begin{align*}
        \mathbb{E}[X] &= \frac{q-p}{(1-q)(1-p)}, \\
        \Var[X] &= \frac{p}{(1-p)^2} + \frac{q}{(1-q)^2}.
    \end{align*}
\end{lem}
}
\begin{proof}
    Direct computation.
\end{proof}

\begin{lem} \label{lem: asymmetric Laplace majorizer in discrete}
    For a log-concave random variable $Y \sim g$, there exists $X$ with mode-centered asymmetric Laplace probability sequence $g_{p,q}$ such that, $g_{p,q}(0) = g(0)$ and
    \[
        Y \prec_{cx} X.
    \]
    Moreover, if $\Var(Y) = \Var(X)$ then $Y \sim g_{p,q}$ as well.
\end{lem} 
\begin{proof}
The proof of the existence of such an $X$ follows as in the proof of \Cref{prop: match a point and convex domination}.  
Now we prove that $Y \sim g_{p,q}$ if $\Var(Y) = \Var(X)$.  
Note that since $Y \prec_{cx} X$ implies $\mathbb{E}Y = \mathbb{E}X$, $\Var(Y) = \Var(X)$ is equivalent to $\mathbb{E} Y^2 = \mathbb{E} X^2. $ 
Moreover as in \Cref{prop: match a point and convex domination},  $g$ is identical to $g_{p,q}$ or there exists $x_1 > x_2$ such that $\phi \coloneqq g_{p,q} - g$ is such that $ \phi(x) < 0$ for $x \in (x_2, x_1)$ and $\phi(x) \geq 0$ for $x \in [x_2, x_1]^c$. 
This gives $\int f \phi d \mu \geq 0$ for $f$ convex, and since we have in particular $\int x^2 \phi(x) d\mu(x) = 0$, 
it follows that $Y$ and $X$ have the same distribution.
\end{proof}
Note $g_{p,q}(0) = g(0) = \frac{(1-p)(1-q)}{1 - pq}$ while $\mu \coloneqq \mathbb{E}[Y] = \mathbb{E}[X] = \frac{q-p}{(1-q)(1-p)}$, which can be solved for $p$ and $q$ as 
\[
    p 
    = \frac{1 - y (\mu+1)}{1 - g(0) (\mu -1)},
        \qquad
    q 
    = \frac{1 + y(\mu-1)}{1 + g(0) (\mu + 1)},
\]
thus we have
\[
    \Var[Y] \leq \Var[X] = \frac{1}{2} \left( \left(\frac 1 {g(0)^2} - 1\right) + \mu^2 \right).
\]
Note that by \Cref{lem: asymmetric Laplace majorizer in discrete} equality in the above equation implies that $Y$ has an asymmetric Laplace distribution.
Collecting the above we have \Cref{thm: discrete variance point inequality}, which we restate for convenience below.

\begin{thm}
    For $Y$ a log-concave random variable on $\mathbb{Z}$, and $n \in \mathbb{Z}$,
    \[
       2\  \Var[Y] \leq \left( \frac{1}{\mathbb{P}^2(Y=n)} - 1 \right) + (\mu - n)^2 .
    \]
    with equality if and only if $Y$ has an asymmetric Laplace distribution with mode $n$.
\end{thm}

Alternatively, one can consider the problem of maximizing the variance under a fixed maximum value $M(Y) = M$, without any other constraint.  
To this end we may without loss of generality assume that $Y$ has $0$ as a mode, and that $\mathbb{E}Y \le 0$.  
By \Cref{lem: asymmetric Laplace majorizer in discrete}, any maximizer of the variance is necessarily a asymmetric Laplace distribution $X$.  
Thus we consider $X$ asymmetric Laplace with mode at $0$  and $q \leq p$.  
Note that $p = q$ for fixed maximum $M$ forces 
$q = \frac{1 -M}{1 + M}$.  Thus we wish to maximize $\Var[X_{p,q}]$ over $q \in [0,\frac{1-M}{1+M}]$ when $p = \frac{ 1 - q - M}{1 - q - Mq}$.  
After algebra,
\[
    \Var[X_{p,q}] 
    = \frac{1}{M^2} - \frac{(1+q)}{M(1-q)} + \frac{2q}{(1-q)^2},
\]
so it suffices to prove
\[
    \frac{d}{dq} \left(- \frac{(1+q)}{M(1-q)} + \frac{2q}{(1-q)^2}\right) 
    = \left(\frac{2}{M(1-q)^3}\right)\big(-(1-M) + (1 + M)q\big)
    < 0
\]
for $0 < q < \frac{1-M}{1+M}$, 
which is immediately verified. 
Taking $q = 0$ gives a the unique maximizer of the variance for log-concave variables with fixed maximum $M$, mode at $0$, and non-positive mean to be the (reflected) geometric distribution.  
Summarizing, for a log-concave random variable with fixed maximum, the geometric distribution, up to translation and reflection, is the unique maximizer of variance. 
Noting that for a geometric distribution $Z$ we have 
\[
    \Var[Z] = \frac{1- M(Z)}{M^2(Z)},
\]
we have thus recovered half of \Cref{thm: discrete piotr theorem}, which we restate below.
\begin{thm} \label{thm: recovering piotrs}
    For $X$ a log-concave random variable on $\mathbb{Z}$,
    \[
        M^2(X) \Var(X) + M(X) \leq 1,
    \]
    with equality if and only if $X$ is, up to translation and reflection, a geometric random variable.
\end{thm}

This approach does work in more generality, however the computations become more difficult.  For example we obtain the following which constitutes the second half of \Cref{thm: discrete piotr theorem}.

{\color{black}\begin{theorem}
    For $X$ a log-concave random variable on $\mathbb{Z}$,
    \[
        M^4(X) \sigma_4(X) + M(X) \big( M^2(X) - 10 M(X) + 18 \big) \leq 9
    \]
    with equality if and only if $X$ is, up to translation and reflection, a geometric random variable.
\end{theorem}
}
We note that the constant terms in Theorem \ref{thm: recovering piotrs} $1=  \ !2$ while the constant term in the above $9 = \ !4$, in agreement with the constants in the continuous case above in  Corollary \ref{cor: maximums to ACM}.  

\begin{proof}[{\color{black}Proof.}]
   For a fixed maximum $M$, by Lemma \ref{lem: asymmetric Laplace majorizer in discrete} it suffices to prove the result for asymmetric Laplace $X_{p,q}$, with $M(X_{p,q}) = M$, $q \in \left[ 0 , \frac{1-M}{1+M} \right]$ and $p = \frac{1 - q -M}{1-q-Mq}$.  It can be computed that,
    \[
    \sigma_4(X) = \sigma_4(q) = \frac{ 9 c_0(q) - 18 c_1(q) M + 10  c_2(q) M^2 - c_3(q )M^3   + 2 q c_4(q) M^4}{ M^4 ( 1 - q)^4},
    \]
    with
    \begin{align*}
        c_0(q) &= (1-q)^4
            \\
        c_1(q) &= (1-q)^3 (1+q)
            \\
        c_2(q) &= (1-q)^2 (1 + 4q + q^2)
            \\
        c_3(q) &= ( 1 + 22q - 22q^3 - q^4)
            \\
        c_4(q) &= ( 1 + 10q + q^2)   
    \end{align*}
    The derivative can be computed and simplified to 
    \[
        \sigma_4'(q) = 2 \frac{ -18 k_0 + 30 k_1 M -  13 k_2 M^2 + k_3 M^3  }{M^3 (1-q)^5}
    \]
    where
    \begin{align*}
        k_0(q) &= (1-q)^3
            \\
        k_1(q) &= (1-q)^2 (1+q)
            \\
        k_2(q) &= 1  + \frac{33}{13} q - \frac{33}{13} q^2 - q^3
            \\
        k_3(q) &= 1 + 23 q + 23q^2 + q^3
    \end{align*}
    Note that $\sigma_4(q) = 0$ if and only if $ -18 k_0 + 30 k_1 M -  13 k_2 M^2 + k_3 M^3 = 0$.  For fixed $M$, this is a cubic polynomial in $q$ with 2 complex roots, and one real root $ q = \frac{1-M}{1+M}$.  Thus to prove that $\sigma_4'(q) < 0 $ for $ q \in (0, \frac{1-M}{1+M})$ it suffices to prove that $\sigma_4'(0) < 0$, which is equivalent to $ - 18 + 30 M - 13 M^2 + M^3 < 0$ for $M \in (0,1)$, which is easily verified.  It follows that for $M(X) = M$, $$ \sigma_4(X) \leq \sigma_4(0) = \frac{9 - 18 M + 10M^2 - M^3}{M^4}.$$
    That the inequality is strict for $X$ that are {\it not} asymmetric Laplace follows by Lemma \ref{lem:two_crossing} and the strict convexity of $f(x) = x^4$, as it provides an asymmetric Laplace $X_{p,q}$ such that $\sigma_4(X) < \sigma_4(X_{p,q})$. That the inequality is actually strict among asymmetric Laplace distributions that are {\it not} geometric follows from the strict negativity of $\sigma_4'(q)$, analogously to Theorem \ref{thm: recovering piotrs}.
\end{proof}

\section{Majorization and derivatives} \label{sec: Majorization}
The results here and generalizations thereof are classical, see \cite{karlin1963generalized}, but we include simple proofs for completeness.  For $n \in \mathbb{N}$ we will employ the notation $[n] = \{0,1, \dots, n\}$.

\begin{proof}[Proof of \Cref{lem:two_crossing}]
    There exists $A$ an affine function such that $\tilde{f} \coloneqq f - A$ is nonnegative outside of $[x_0, x_1]$ and nonpositive inside of $(x_0,x_1)$.  Then
    \[
        \int \tilde{f} \phi d\mu = \int f \phi d\mu - \int A \phi d\mu
    \]
    but $A(x) = mx +b$ for constants $ m$ and $b$, so 
    \[
        \int A \phi d\mu = m \int \psi d \mu + b \int \phi d \mu = 0.
    \]
    Thus we have
    \[
        \int f \phi d\mu = \int \tilde{f} \phi d\mu
    \]
    and by construction $\tilde{f} \phi \geq 0$.  In the case that there exists equality for $f$ strictly convex, 
    then the construction above of $\tilde{f}$ is strictly convex, and hence is non-zero away from $x_0$ and $x_1$.
    Since $\int |\tilde{f} \phi|  d\mu = \int \tilde{f} \phi d\mu = 0$, we have that $\tilde{f} \phi = 0$ $\mu$-almost surely.
    But since $\tilde{f}$ is non-zero away from the $x_i$, we have $\phi = 0$ $\mu$-almost surely away from the $x_i$.  
    However, by the assumption $0 = \int \phi(x) d\mu = \int x \phi(x) d\mu(x)$, we have $\Phi u = 0$ where 
    \begin{align*}
         \Phi \coloneqq \begin{pmatrix}
            \phi(x_0) & \phi(x_1) \\
            x_0 \phi(x_0) & x_1 \phi(x_1)
            \end{pmatrix} \hbox{ and } u \coloneqq \begin{pmatrix}
                \mu \{x_0\} \\
                \mu \{x_1\}
                \end{pmatrix}.
    \end{align*}
    If $u = 0$, then $\phi = 0$ $\mu$-almost surely and we are done.  
    If $u \neq 0$, then the rows of $\Phi$ are linearly dependent.  
    Since $x_0 \neq x_1$, this is only possible if at least one of $\phi(x_0)$ and $\phi(x_1)$ is zero.  
    For concreteness, say $\phi(x_0) = 0$:
    then, considering the first entry of $\Phi u$ we have $\phi(x_1) \mu\{x_1\} = 0$.  
    Thus either $\phi(x_1)$ or $\mu\{x_1\}  =0$; 
    in either case we are done.
\end{proof}

We note that the proof of \Cref{lem:one_crossing} is similar and simpler; taking $\tilde{f}(x) = f(x) - f(x')$ one can proceed analogously to the above.

\bigskip

    Now let us consider a generalization of these ideas to the case of $\phi$ with $n$ crossings, orthogonal (with respect to $\mu$) to all polynomials of degree $k < n$.  
    To this end, we denote the $k$-th derivative of a function $g$ by $g^{(k)}$, with the convention that $g^{(0)} = g$.  
    We will consider the class of functions whose $n$-th derivative strictly convex.
    That is, for $n \geq 0$,
    \[
        \mathcal{A}_n \coloneqq \{g: \mathbb{R} \to \mathbb{R} \mid g^{(n)} \mbox{ is strictly convex } \} .
    \]
    By convention we consider $\mathcal{A}_{-1}$ to be the  the class of strictly increasing functions, with the heuristic that an anti-derivative of a strictly increasing function will be strictly convex.  Similarly, we adopt the convention that $g^{(-1)}$ is convex when $g$ is increasing.  We will need the following elementary lemma.

\begin{lem} \label{lem: sign switching functions}
    For $n \geq 1$ and $ g \in \mathcal{A}_{n-2}$, $g$ has at most $n$ zeros.  
    Moreover, if $g$ has exactly $n$ zeros $x_1 < x_2 < \cdots < x_n$, then
    \begin{align} \label{eq: poly g inequality}
        g(x) \prod_{k=1}^n (x - x_k) \geq 0.
    \end{align} 
\end{lem}
Note that the inequality \eqref{eq: poly g inequality} is necessarily strict for $x$ such that $g(x) \neq 0$.  

\begin{proof}
    The $n=1$ case follows trivially. 
 When $n=2$, $g$ strictly convex with zeros $x_1$ and $x_2$ implies
    \begin{equation} \label{eq: smaller than 0}
        g((1-t) x_1 + t x_2) < 0
    \end{equation}
    for $t \in (0,1)$ while
      \begin{equation} \label{eq: bigger than zero}
        g((1-t) x_1 + t x_2) > 0
    \end{equation}
    for $t \in (-\infty, 0) \cup (1, \infty)$, hence $g$ can have at most two zeros.  
    Moreover, the observations \eqref{eq: smaller than 0} and \eqref{eq: bigger than zero} imply that $g(x)(x - x_1)(x- x_2) \geq 0 $.

     Proceeding by induction for $n \geq 3$, suppose $g^{(n-2)} = (g')^{(n-3)} $ is convex, so that  $g'$ has no more than $n-1$ zeros.  By the mean value theorem $g$ has no more than $n$ zeros.  
     In the case that $g$ has $n$ distinct zeros $x_1 < \cdots < x_n$, Rolle's theorem gives the existence of $z_1, \dots, z_{n-1}$ zeros of $g'$ interlacing the zeros of $g$ in the sense that $$x_i < z_i < x_{i+1}.$$ 
     By the already proven first half of this theorem, $g'$ has no more than $n-1$ zeros, and 
     by induction
     \[
        g'(x) \prod_{k=1}^{n-1} (x- z_k) \geq 0.
     \]  
     In particular, $g'(x) > 0$ for $x > z_k$.  By the interlacing of zeros $g'$ can be bounded away from zero at the zeros of $g$ and hence $g$ necessarily changes sign at its zeros. 
     It follows that
     \[
        g(x) \prod_{k=1}^n (x- x_k)
     \]
     is either non-negative or non-positive.  To check its sign, take $x > x_n > z_{n-1}$ and observe that by the mean value theorem there exists $ x^* \in [x_n , x] $ such that $g(x) = g'(x^*)(x - x_n) + g(x_n)$.  The latter is positive and the conclusion follows.
\end{proof}

\begin{lem} \label{lem: modulo polynomial}
    Given $f \in \mathcal{A}_n$ and real numbers $x_1 < x_2 < \cdots < x_n$, there exists a polynomial $P(x) = \sum_{j=0}^d p_j x^j$ of degree $d < n$ such that 
    \[
        ( f - P)(x) \prod_{k=1}^n (x - x_k) \geq 0.
    \]
\end{lem}
\begin{proof}

    Let 
    \[P(x)\coloneqq\sum_{k=1}^n f(x_k) \prod_{j\neq k}\frac{x-x_j}{(x_k-x_j)}.\]
    and observe $P(x_k)=f(x_k)$ for each $k$, and that $P^{(n-2)}$ is affine since $\deg(P)=n-1$. Thus $(f-P)^{(n-2)}$ is strictly convex and has exactly $n$ distinct zeros occurring at each of the $x_k$.

    The result follows by letting $g=f-P$ in Lemma \ref{lem: sign switching functions}.
\end{proof}

\begin{thm} \label{thm: Karlin-Novikoff criteria}
    Let $\mu$ be a Borel measure on $\mathbb{R}$. Suppose that $\phi: \mathbb{R} \to \mathbb{R}$ is such that there exists $x_1 < x_2 < \cdots < x_n$ such that $\phi(x) \prod_{k=1}^n (x - x_k) \geq 0$. Suppose that $x^k \phi(x)$ belongs to $L_1(\mu)$ for $k \leq n$ with
    \[
        \int x^k \phi(x) d\mu(x) = 0
    \]
    for $ k \leq n-1$ and that $f^{ (n-2)}$ is convex. Then
    \[
        \int f \phi d\mu \geq 0.
    \]
\end{thm}

\begin{proof}
    Applying \Cref{lem: modulo polynomial} to $f_\varepsilon(x) \coloneqq f(x) + \varepsilon x^n$ and the $x_i$ we obtain a polynomial $P$ of degree $n-1$ such that $(f_\varepsilon-P) \phi \geq 0$.  Indeed, $(f_\varepsilon - P)(x_k) = 0$ by construction, while
    \[
        \left(\prod_{k=1}^{n}(x - x_k) \right)^2 (f_\varepsilon - P)\phi(x) \geq 0
    \]
    for $x \notin \{x_k\}_{k=1}^n$.  Since $P$ is a degree $n-1$ polynomial, $\int P \phi d\mu = 0$. Thus
\begin{align*}
    \int f_\varepsilon \phi d\mu = \int (f_\varepsilon -P) \phi d\mu  + \int P \phi d\mu \geq 0.
\end{align*}
Taking $\varepsilon \to 0$ completes the proof.
\end{proof}
We do not pursue the minimal hypothesis in the above. Weaker moment conditions can be approached by more delicate approximations of $f$. In the context of log-concave random variables the existence of moments of all orders is implied. We conclude with a reformulation of the above for random variables.
\begin{cor} \label{cor: n crossings for random variables}
    Let $X_1$ and $X_2$ be random variables with $\mathbb{E}|X_i|^n < \infty$ 
    and densities $\phi_1$ and $\phi_2$ with respect to a common reference measure $\mu$.  
    Let $ x_1 < x_2 < \cdots < x_n$ be such that $(\phi_2 - \phi_1)(x) \prod_{k=1}^n (x_k - x) \geq 0$,
    and suppose that
    \(
        \mathbb{E}[X^k_1] = \mathbb{E}[X_2^k] 
    \)
    for all $k \leq n-1$. 
    Then~$X_1 \prec_n X_2$.
\end{cor}
\begin{proof}
    Apply Theorem \ref{thm: Karlin-Novikoff criteria} to $\phi \coloneqq \phi_2 - \phi_1$.
\end{proof}

\bibliographystyle{plainurl}
\bibliography{VMC}

\end{document}